\documentclass{amsart}
\usepackage{amssymb}
\usepackage{amsmath}
\usepackage{amsthm}
\usepackage{mathtools}
\usepackage{ifpdf}

\ifpdf
\usepackage[pdftex]{hyperref}
\else
\usepackage[hypertex]{hyperref}
\fi

\hypersetup{
citecolor=green,
menucolor=red,
citebordercolor=0 1 0,
linkbordercolor=1 0 0,
pdfpagemode=UseOutlines,
pdftitle={Multi-Speeds solitary waves solutions for nonlinear Schrodinger systems},
pdfauthor={Isabella Ianni <isabella.ianni@unina2.it>, Stefan Le Coz <slecoz@math.univ-toulouse.fr>},
pdfsubject={Multi-Speeds solitary waves solutions for nonlinear Schrodinger systems},
pdfkeywords={solitary waves, nonlinear Schroedinger systems} 
}

\newtheorem{thm}{Theorem}
\newtheorem{prop}{Proposition}
\newtheorem{lem}[prop]{Lemma}

\DeclarePairedDelimiter{\norm}{\lVert}{\rVert}
\DeclarePairedDelimiter{\abs}{\lvert}{\rvert}

\newcommand{\psld}[2]{\left( #1,#2 \right)_{2}}

\newcommand{\dual}[2]{\left\langle #1,#2 \right\rangle}

\newcommand{\eps}{\varepsilon}

\newcommand{\N}{\mathbb{N}}
\newcommand{\R}{\mathbb{R}}

\newcommand{\hu}{H^1(\R^d)}

\renewcommand{\leq}{\leqslant}
\renewcommand{\geq}{\geqslant}

\DeclareMathAlphabet{\mathpzc}{OT1}{pzc}{m}{it}
\renewcommand{\Re}{\mathcal R\!\mathpzc{e}}
\renewcommand{\Im}{\mathcal I\!\mathpzc{m}}

\begin{document}

\title[Multi-Speeds Solitary Waves for NLS systems]{Multi-Speeds Solitary Waves Solutions For Nonlinear Schr\"odinger Systems}

\author[I.~Ianni]{Isabella Ianni}
\author[S.~Le Coz]{Stefan Le Coz}

\address{
Dipartimento di Matematica - Facolt\`a di Scienze MM.FF.NN.
\newline\indent
Seconda Universit\`a di Napoli
\newline\indent
Viale Lincoln 5, 81100 Caserta
\newline\indent
Italia}
\email{isabella.ianni@unina2.it}

\address{Institut de Math\'ematiques de Toulouse,
\newline\indent
Universit\'e Paul Sabatier
\newline\indent
118 route de Narbonne, 31062 Toulouse Cedex 9
\newline\indent
France}
\email{slecoz@math.univ-toulouse.fr}

\thanks{
Research supported in part by  ANR project ESONSE
}

\date{\today}

\begin{abstract}
We prove the existence of a new type of solutions to a nonlinear Schr\"odinger system. These solutions, which we call \emph{multi-speeds solitary waves}, are behaving at large time as a couple of scalar solitary waves traveling at different speeds. 
The proof relies on the construction of approximations of the multi-speeds solitary waves by solving the system  backwards in time and using energy methods to obtain uniform estimates.
\end{abstract}

\maketitle

\section{Introduction}

We consider the following nonlinear Schr\"odinger system:
\begin{equation}\label{eq:nls}
\left
\{\begin{array}{lr}
i\partial_t u_1+\Delta u_1+\mu_1|u_1|^2u_1+\beta|u_2|^2u_1=0,\\
i\partial_t u_2+\Delta u_2+\mu_2|u_2|^2u_2+\beta|u_1|^2u_2=0,
\end{array}
\right.
\end{equation}
where  for $j=1,2$ we have $u_j:\mathbb R\times\mathbb R^d\rightarrow\mathbb C$, $d=1,2,3$, $\mu_j>0$, and $\beta\in\R\setminus\{0\}$.

This type of systems appears in various physical settings, of which we give now three examples. 

When $d=\mu_1=\mu_2=\beta=1$, the system \eqref{eq:nls} is sometimes called \emph{Manakov system}, as it was examined by Manakov \cite{Ma74} as an asymptotic model for the propagation of the electric field in a wageguide. With this specific choice of parameters, the system is completely integrable and can be solved by means of the inverse scattering transform. Such analysis is performed in details in the book \cite{AbPrTr04}, which contains also many examples of physical situations where \eqref{eq:nls} is used.

Later on, \eqref{eq:nls} was derived to model the propagation of light in an optical fiber when taking into account polarization of light and birefringence of the fiber, see e.g.  \cite{Ab07}.  In this case $d=\mu_1=\mu_2=1$ and the parameter $\beta$, which measure the strength of  the XPM (cross phase modulation) interaction, varies depending on the nature of the fiber (e.g. $\beta=2$ for dual-core fibers or $\beta=2/3$ for single-core fibers).

In higher dimension $d=3$, \eqref{eq:nls} can model the interaction of two Bose-Einstein condensates of atoms in different spin states (see e.g \cite{Es97}).  In this case, if $N$ denotes the number of atoms in the $j$-st condensate and $a_{jk}$ is a factor proportional to the scattering length between a $j$-species atom and a $k$-species atom ($a_{jk}$ may be positive or negative, depending if the collision between particles results into an attractive of repulsive interaction), the parameters of \eqref{eq:nls} stands for $\mu_j=(N-1)a_{jj}$ and $\beta=Na_{jk}$. The trapping potential is turned off to model the expansion of the condensates in experiments.

From the mathematical point of view, there has been recently an increasing interest for
\eqref{eq:nls}  and its stationary versions. We give only a few samples of the mathematical studies around \eqref{eq:nls}.  
As mentioned before, the system is completely integrable in the Manakov case, but any modification of the parameters breaks integrability and the analysis of the dynamics of \eqref{eq:nls} in non-integrable cases is largely open. A lot of recent studies (see e.g. \cite{AmCo07,LiWe05,MaMoPe06,Si07,TeVe09,WeWe08}) are concerned with the existence of standing waves solutions  for various ranges of parameters $\mu_1,\mu_2,\beta$. The stability of  such standing waves was also investigated in various cases (see, among many others, \cite{CoCoOh09,MaMoPe10,Oh96}). 

In this work, we want to investigate the existence of solutions to \eqref{eq:nls}  where each component behaves like a soliton, as we explain precisely now. 

When $u_1\equiv0$ or $u_2\equiv0$, the system \eqref{eq:nls} reduces to the scalar  Schr\"odinger equation 
 \begin{equation}\label{eq:nls-scalar-intro}
i\partial_tu+\Delta u+\mu|u|^2u=0.
\end{equation}
It is well known that \eqref{eq:nls-scalar-intro} admits \emph{solitary waves} (see \cite{La80,Sc86}), which are solutions with a fixed profile, possibly rotating and traveling on a line (see Theorem \ref{thm} and Section \ref{sec:nls-scalar} for more details). 
 If $R_1$ denotes a solitary wave solution to \eqref{eq:nls-scalar-intro}, then $(R_1,0)^\intercal$ is trivially a solution to \eqref{eq:nls}. If $R_2$ is another solitary wave solution to \eqref{eq:nls-scalar}, then, due to the nontrivial interaction $\beta\neq 0$, the couple $(R_1,R_2)^\intercal$  has no reason to be a solution to \eqref{eq:nls}. Nevertheless, our goal in this paper is to exhibit solutions of \eqref{eq:nls} behaving in large time  like a couple of solitary waves $(R_1,R_2)^\intercal$, provided the relative speed of the solitary waves is large enough. We call such solutions \emph{multi-speeds solitary waves}. To our knowledge, this is the first time that such solutions are exhibited for non-integrable Schr\"odinger systems. Our main result is the following.

\begin{thm}\label{thm}
For $j=1,2$, let $\omega_j>0,\gamma_j\in\R$, $x_j,v_j\in \R^d$,  and $\Phi_j\in\hu$ solution to
\begin{equation}\label{eq:snls-basic-intro}
-\Delta \Phi_j+\Phi_j-|\Phi_j|^2\Phi_j=0,\qquad\Phi_j\in\hu.
\end{equation}
Define  
\begin{gather}
R_j(t,x):=e^{i(\omega_j t-\frac{|v_j|^2t}{4}+\frac{1}{2}v_j\cdot x+\gamma_j)}\sqrt{\frac{\omega_j}{\mu_j}}\Phi_j\big(\sqrt{\omega_j} (x-v_jt-x_j)\big),\label{eq:soliton-j}\\
v_\star:=|v_1-v_2|,\qquad \omega_\star:=\frac{1}{4}\min\{\omega_1,\omega_2\}.\label{eq:omega-v}
\end{gather}
There exists $v_\sharp>0$ such that if $v_\star>v_\sharp$, then there exists $T_0\in\R$ and a \emph{multi-speeds solitary wave}   $(u_1,u_2)^\intercal$ solution of \eqref{eq:nls} defined on $[T_0,+\infty)$ such that for all $t\in[T_0,+\infty)$ the following holds:
\begin{equation}\label{eq:thm1}
\norm*{\begin{pmatrix}u_1(t)\\u_2(t)\end{pmatrix}-\begin{pmatrix}R_1(t)\\R_2(t)\end{pmatrix}}_{H^1\times H^1}\leq e^{-\sqrt{\omega_\star}v_\star t}.
\end{equation}
\end{thm}

The strategy of the proof of Theorem~\ref{thm} is inspired by the one developed for the study of multi-solitons for scalar nonlinear Schr\"odinger equations in  \cite{CoLe11,CoMaMe11,MaMe06,Me90}. The idea is to solve \eqref{eq:nls} backward in time taking as final data at the final time $T^n$ a couple of solitary waves, where $T^n$ is an increasing sequence of time. 
In this way, we define a sequence of solutions to \eqref{eq:nls} which are approximated multi-speeds solitary waves. Then the proof relies on two main steps. First we show that the approximate solutions satisfy the required estimate \eqref{eq:thm1} on a sequence of time-intervals $[T_0,T^n]$, \emph{with $T_0$ independent of $n$} (see Proposition \ref{prop:uniform}). Then we prove that the sequence of initial data obtained at $T_0$  is compact (see Proposition \ref{prop:compactness}). Therefore, we can extract an initial data giving rise to a solution of \eqref{eq:nls} which satisfies the conclusion of Theorem \ref{thm}. 

Our approach is very flexible and can probably be extended to many other situations. We do not need many of the technical features present in \cite{CoLe11,CoMaMe11,MaMe06,Me90} like modulation theory or localization procedures. Neither do we require any assumptions on the attractiveness ($\beta<0$) or repulsiveness ($\beta>0$) of the coupling, or on the strength of the nonlinearities  $\mu_1,\mu_2$. Whereas it is common when working with solitary waves to consider only ground states profiles $\Phi_j$ of \eqref{eq:snls-basic-intro}, in our case, as in \cite{CoLe11}, the profiles  can be ground states or excited states. Our only limitation is the assumption on large relative speed $v_\star$, which is due to technical restrictions when proving the uniform estimates (see Section \ref{sec:uniform}).

The rest of the paper is divided as follows. In Section \ref{sec:nls-scalar}, we gather some useful facts about scalar nonlinear Schr\"odinger equations and their solitary waves. Then in Section \ref{sec:construction} we prove the existence of multi-speeds solitary waves assuming uniform estimates and a compactness result. The proof of the uniform estimates and the compactness result are given in Sections \ref{sec:uniform} and \ref{sec:compactness}.

\subsection*{Notations}

Before going further, we precise some notations.
The norms of $L^{p}(\R^d)$ spaces will be denoted by $\norm{\cdot}_{L^p}$ and the norm of $\hu$ by $\norm{\cdot}_{H^1}$.
The spaces $L^2(\R^d)\times L^2(\R^d)$ and $\hu\times\hu$ are endowed with the norms 
\[
\norm*{\begin{pmatrix}u_1\\u_2\end{pmatrix}}_{L^2\times L^2}=\sqrt{\norm{u_1}^2_{L^2}+\norm{u_2}^2_{L^2}},
\qquad
\norm*{\begin{pmatrix}u_1\\u_2\end{pmatrix}}_{H^1\times H^1}=\sqrt{\norm{u_1}^2_{H^1}+\norm{u_2}^2_{H^1}}.
\]
When writing vectors inside the text, we will use the superscript $\intercal$ to denote the transpose of a vector, that is: 
$
(u_1,u_2)^\intercal=\begin{pmatrix}u_1\\u_2\end{pmatrix}.$
The derivative with respect to the time $t$ will be denoted either by $\frac{\partial}{\partial t}$ or simply~$\partial_t$.
Throughout the paper the letter $C$ will denote various positive constants whose exact values may change from line to line but are of no importance for the analysis.

\section{Scalar solitary waves}\label{sec:nls-scalar}

In this section we summarize the results on scalar solitary waves that we will need for the proof of Theorem \ref{thm}. For more on scalar Schr\"odinger equations, the reader can refer to  \cite{Ca03,SuSu99,Ta06} and the references cited therein. Consider the scalar Schr\"odinger equation 
\begin{equation}\label{eq:nls-scalar}
i\partial_tu+\Delta u+\mu_0|u|^2u=0,
\end{equation}
where $\mu_0$ is a positive constant. 

The energy, mass and momentum, defined as follows, are conserved along the flow of \eqref{eq:nls-scalar}. 
\begin{align*}
E(u,\mu_0)&:=\frac{1}{2}\norm{\nabla u}_{L^2}^2-\frac{\mu_0}{4}\norm{u}_{L^4}^4,\quad
M(u):=\frac12\norm{u}^2_{L^2},\quad
P(u):=\frac12\Im\int_{\mathbb R^d} u\nabla \bar udx.
\end{align*}

A basic solitary wave $u$ is a solution of \eqref{eq:nls-scalar} of the form $u(t,x)=\frac{e^{it}}{\sqrt{\mu_0}}\Phi(x)$, where $\Phi$ is a solution of 
\begin{equation}\label{eq:snls-basic}
-\Delta \Phi+\Phi-|\Phi|^2\Phi=0,\qquad\Phi\in\hu.
\end{equation}
The existence and properties of solutions to equations of the type \eqref{eq:snls-basic} are well-known (see e.g. the fundamental work of Berestycki and Lions \cite{BeLi83-1,BeLi83-2}). All solutions to \eqref{eq:snls-basic} are smooth and exponentially decreasing. Precisely, for all $\eta<1$, for all solutions $\Phi$ to \eqref{eq:snls-basic}, and for all $x\in\R^d$, there exists $C_\Phi>0$ such that the following estimate holds:
\[
|\Phi(x)|+|\nabla\Phi(x)|\leq C_\Phi e^{-\eta |x|}.
\]
Equation \eqref{eq:snls-basic} admits a unique \emph{ground state}, i.e. a positive and radial solution which minimizes among all solutions the \emph{action} $S:=E(\cdot,1)+M$. In dimension~$d\geq 2$, there exist also infinitely many other solutions called \emph{excited states}. Apart when $d=1$, the classification of solutions to \eqref{eq:snls-basic} is still an active research area. Classification of radial solutions was completed recently in the works \cite{CoGaYa09,CoGaYa11}. Among non radial solutions we mention the vortices, which were first constructed by Lions \cite{Li86}. In dimension~$2$, a vortex is a solution of \eqref{eq:snls-basic} of the form $\Phi(\rho,\theta)=e^{im\theta}\Psi(\rho)$ where $(\rho,\theta)$ are polar coordinates and $m\in\R$. In this work, we treat any type of solutions to \eqref{eq:snls-basic}.

Invariances by scaling, translation, phase shift and galilean transform generate a $(2d+2)$-parameters family of solitary waves solutions to \eqref{eq:nls-scalar}. Precisely, let $\omega_0>0,\gamma_0\in\R$, $x_0,v_0\in \R^d$, and take $\Phi_0\in\hu$ a solution to \eqref{eq:snls-basic}. Then $R_0$ defined by
\begin{equation}\label{eq:soliton}
R_0(t,x):=e^{i(\omega_0 t-\frac{|v_0|^2}{4}+\frac{1}{2}v_0\cdot x+\gamma_0)}\sqrt{\frac{\omega_0}{\mu_0}}\Phi_0\big(\sqrt{\omega_0} (x-v_0t-x_0)\big)
\end{equation}
is a solution to \eqref{eq:nls-scalar}. Note that, for $t$ fixed, $R_0(t,\cdot)$ is a critical point of the functional $S_0$ defined by
\begin{equation}\label{eq:def-S}
S_0:=E(\cdot,\mu_0)+\left(\omega_0+\frac{|v_0|^2}{4}\right)M+v_0\cdot P.
\end{equation}
Coercivity properties of linearizations of $S_0$-like functionals will play an important role in our analysis. We define the linearized action $H_0$ for $t\in\R$ and $\eps\in\hu$ by
\begin{equation}\label{eq:def-H}
H_0(t,\eps):=\dual{S_0''(R_0(t))\eps}{\eps}.
\end{equation}

\begin{lem}[Scalar Coercivity]\label{lem:coercivity}
Take $\omega_0>0,\gamma_0\in\R$, $x_0,v_0\in \R^d$, $\Phi_0\in\hu$ a solution to \eqref{eq:snls-basic} and let $R_0$ be the solitary wave solution of \eqref{eq:nls-scalar} given by \eqref{eq:soliton}, $S_0$ and $H_0$ the functionals given by \eqref{eq:def-S}-\eqref{eq:def-H}. Then there exists $c_0>0$, $\nu_0\in\N$, and a family of normalized functions $\{\xi_0^k\in L^2(\R^d); \norm{\xi_0^k}_{L^2}=1, k=1,...,\nu_0\}$ such that for all $t\in\R$ and for all $\eps\in\hu$ we have
\[
c_0\norm{\eps}_{H^1}^2\leq H_0(t,\eps)+\sum_{k=1}^{\nu_0}\psld{\eps}{\xi_0^k(t)}^2,
\]
where by $\xi_0^k(t)$ we denote the functions defined by
\[
\xi_0^k(t)(x):=e^{i(\omega_0 t-\frac{|v_0|^2}{4}+\frac{1}{2}v_0\cdot x+\gamma_0)}\sqrt{\frac{\omega_0}{\mu_0}}\xi_0^k\big(\sqrt{\omega_0} (x-v_0t-x_0)\big).
\]
\end{lem}

\begin{proof}[Sketch of proof]
The result being classical we only recall the main arguments. Consider $\Phi$ a real solution of \eqref{eq:snls-basic}. Then $\Phi$ is a critical point of the functional $S=E(\cdot,1)+M$. 
For $\eps\in\hu$, the functional $\dual{S''(\Phi)\eps}{\eps}$ can be  decomposed by writing
\begin{equation*}
\dual{S''(\Phi)\eps}{\eps}=\dual{L_+\Re(\eps)}{\Re(\eps)}+\dual{L_-\Im(\eps)}{\Im(\eps)},
\end{equation*}
where  $L_+,L_-$ are two self-adjoint linear operators defined by:
\begin{align*}
L_+& =-\Delta+1-3|\Phi|^2,\\
L_-& =-\Delta+1-|\Phi|^2.
\end{align*}
The operators $L_+$ and $L_-$ are self adjoint compact perturbations of $-\Delta+1$, hence their spectrums lie on the real line and consist of essential spectrum on $[1,+\infty)$  and a finite number of eigenvalues on $(-\infty,\eta]$ for any $\eta<1$. Hence, there exists $c_0>0$, $\nu_0\in\N$ corresponding to the number of non-positive eigenvalues of $L_+$ and $L_-$ (counted with multiplicity) and a family of normalized eigenfunctions $\{\xi_0^k\in L^2(\R^d); \norm{\xi_0^k}_{L^2}=1, k=1,...,\nu_0\}$ such that   
\[
c_0\norm{\eps}_{H^1}^2\leq \dual{S''(\Phi)\eps}{\eps}+\sum_{k=1}^{\nu_0}\psld{\eps}{\xi_0^k}^2.
\]
The conclusion of the Lemma~follows by extending the arguments to complex-valued $\Phi$ and applying scaling, phase shift, translations and galilean transform (see \cite{CoLe11,MaMe06} for details). 
\end{proof}

\section{Construction of the solution}\label{sec:construction}
 
Starting from now and for the rest of the paper we fix for $j=1,2$ a set of parameters $\omega_j>0,\gamma_j\in\R$, $x_j,v_j\in \R^d$,  and $\Phi_j\in\hu$ solution to \eqref{eq:snls-basic-intro}. Let  $R_j$ denote the corresponding solitary wave defined in \eqref{eq:soliton-j}, $v_\star$ the relative speed and $\omega_\star$ the minimal frequency, both defined in \eqref{eq:omega-v}.

Before starting the proof, we need some preliminaries on the local well-posedness of \eqref{eq:nls}. In our setting, local well-posedness follows from classical arguments of the local Cauchy theory for Schr\"odinger equations (see e.g. \cite[Remark 3.3.12]{Ca03} and \cite{CaWe90}). Precisely, for any $0<\sigma\leq 1$ such that $2<\frac{4}{d-2\sigma}$ or $\sigma=1$ and for any initial data $(u_1^0,u_2^0)^\intercal\in H^\sigma(\R^d)\times H^\sigma(\R^d)$ there exist $T_\star,T^\star>0$ and a solution  to \eqref{eq:nls} $(u_1,u_2)^\intercal\in\mathcal C\big((-T_\star,T^\star), H^\sigma(\R^d)\times H^\sigma(\R^d)\big)$ such that  
$(u_1(0),u_2(0))^\intercal=(u_1^0,u_2^0)^\intercal$. If in addition $(u_1^0,u_2^0)^\intercal\in\hu\times\hu$, then the solution also belongs to 
$\mathcal C^1\big((-T_\star,T^\star), H^{-1}(\R^d)\times H^{-1}(\R^d)\big)$ and the \emph{blow-up alternative} holds, that is if $T^\star<+\infty$ (resp. $T_\star<+\infty$) then 
\[
\lim_{t\to T^\star} \norm*{\begin{pmatrix}u_1(t)\\u_2(t)\end{pmatrix}}_{H^1\times H^1}=+\infty, \quad\left(\text{resp. }\lim_{t\to -T_\star} \norm*{\begin{pmatrix}u_1(t)\\u_2(t)\end{pmatrix}}_{H^1\times H^1}=+\infty\right).
\]
In the sequel, we shall mainly work with the scalar energy $E(\cdot,\mu)$, momentum $P$ and masse $M$ defined in Section \ref{sec:nls-scalar}, but we remark here that the system \eqref{eq:nls} admits its own conservation laws. Precisely, the total energy $\mathcal E$, the total momentum $\mathcal P$  (defined as follows)  and the masses $M$ of each component are conserved quantities for the $\hu$-flow  of \eqref{eq:nls}:
\begin{gather}
\mathcal{E}
\begin{pmatrix}
u_1(t)\\u_2(t)\end{pmatrix}
:=E(u_1(t),\mu_1)+E(u_2(t),\mu_2)-\frac{\beta}{2}\int_{\mathbb R^d}|u_1(t)|^2|u_2(t)|^2=
\mathcal E \begin{pmatrix}
u_1^0\\u_2^0\end{pmatrix},\label{eq:conservE}
\\
\mathcal P\begin{pmatrix}
u_1(t)\\u_2(t)\end{pmatrix}:=P(u_1(t))+P(u_2(t))=\mathcal P \begin{pmatrix}
u_1^0\\u_2^0\end{pmatrix},\label{eq:conservP}
\\
M(u_1(t))=M(u_1^0), \qquad M(u_2(t))=M(u_2^0).\label{eq:conservM}
\end{gather}

We can now define a sequence of approximated multi-speeds solitary waves. Let $T^n\in\R$ be an increasing sequence of times such that $\lim_{n\to+\infty}T^n=+\infty$. For each $n\in\N$, let $(u^n_1,u^n_2)^\intercal$ be the solution of \eqref{eq:nls} defined on the interval $(T_n,T^n]$ and such that the \emph{final data} satisfy $(u^n_1(T^n),u^n_2(T^n))^\intercal=(R_1(T^n),R_2(T^n))^\intercal$. We will prove that there exists some $T_0$ independent of $n$ such that for every $n$ large enough $(u^n_1,u^n_2)^\intercal$ is defined on $[T_0,T^n]$ and is close to $(R_1,R_2)^\intercal$. More precisely, we have the following proposition, which will be proved in Section \ref{sec:uniform}.
\begin{prop}[Uniform estimates]\label{prop:uniform}
There exists $v_\sharp$ such that if $v_\star>v_\sharp$, then the following holds. 
There exists $T_0\in\R$, and $n_0\in\N$ such that for all $n\geq n_0$ and for all $t\in[T_0,T^n]$ the following estimate is satisfied:
\begin{equation}\label{eq:estimate}
\norm*{\begin{pmatrix}u^n_1(t)\\u^n_2(t)\end{pmatrix}-\begin{pmatrix}R_1(t)\\R_2(t)\end{pmatrix}}_{H^1\times H^1} \leq 
 e^{-\sqrt{\omega_\star}v_\star t}.
\end{equation}
\end{prop}
As $T^n$ goes to $+\infty$, the sequence  $(u^n_1,u^n_2)^\intercal$ provides a better and better approximation of a multi-speeds solitary wave. What remains to show is the convergence of this sequence. Due to local well-posedness and uniform estimates, the main issue is to obtain the convergence of the sequence of  initial data $(u^n_1(T_0),u^n_2(T_0))^\intercal$. This is the object of the following proposition, which will be proved in Section \ref{sec:compactness}.
\begin{prop}[Compactness]\label{prop:compactness}
There exists $(u^0_1,u^0_2)^\intercal\in \hu\times\hu$ such that, possibly for a subsequence only, $(u^n_1(T_0),u^n_2(T_0))^\intercal\to(u^0_1,u^0_2)^\intercal$ strongly in $H^s(\R^d)\times H^s(\R^d)$ for any $s\in[0,1)$ when $n\to+\infty$.
\end{prop}
We can now prove Theorem \ref{thm}.
\begin{proof}[Proof of Theorem \ref{thm}]
Let $(u^0_1,u^0_2)^\intercal$ be the initial data given by Proposition \ref{prop:compactness} and let $(u_1,u_2)^\intercal$ be the solution to \eqref{eq:nls} on $[T_0,T^\infty)$  with initial data $(u_1(T_0),u_2(T_0))^\intercal=(u^0_1,u^0_2)^\intercal$. We show that $T^\infty=+\infty$ and that $(u_1,u_2)^\intercal$ fulfils the conclusions of Theorem \ref{thm}. From Proposition \ref{prop:compactness}, the local well-posedness theory for \eqref{eq:nls}, and the boundedness in $\hu\times\hu$ (implied by Proposition \ref{prop:uniform}), we  have for  $t\in[T_0,T^\infty)$ the convergences
\[
\begin{pmatrix}u^n_1(t)\\u^n_2(t)\end{pmatrix}
\begin{array}{c}
\xrightarrow{H^s\times H^s}\\[-2mm]
\xrightharpoondown[H^1\times H^1]{}
\end{array}
\begin{pmatrix}u_1(t)\\u_2(t)\end{pmatrix},
\]
where the convergence is taken strongly in $H^s(\R^d)\times H^s(\R^d)$ for any $0\leq s<1$ and weakly in $\hu\times\hu$. Consequently, we can estimate for all $t\in[T_0,T^\infty)$: 
\[
\norm*{\begin{pmatrix}u_1(t)\\u_2(t)\end{pmatrix}
-\begin{pmatrix}R_1(t)\\R_2(t)\end{pmatrix}
}_{H^1\times H^1}
\leq
\liminf_{n\to+\infty}\norm*{\begin{pmatrix}u^n_1(t)\\u^n_2(t)\end{pmatrix}
-\begin{pmatrix}R_1(t)\\R_2(t)\end{pmatrix}
}_{H^1\times H^1}
\leq
e^{-\sqrt{\omega_\star}v_\star t}.
\]
In particular, this implies that $(u_1(t),u_2(t))^\intercal$ is bounded in $\hu\times\hu$ on $[T_0,T^\infty)$. Hence, the blow-up alternative implies that $T^\infty=+\infty$ and therefore $(u_1,u_2)^\intercal$ satisfies the conclusions of Theorem \ref{thm}.
\end{proof}

\section{Uniform Estimates}\label{sec:uniform}

In this section, we prove Proposition \ref{prop:uniform}.
From the local well-posedness theory, estimate \eqref{eq:estimate} always holds on some short interval around $T^n$. The goal of the following Lemma~is to allow us to stretch this interval up to the interval $[T_0,T^n]$. 
\begin{lem}[Bootstrap]\label{lem:bootstrap}
There exists $v_\sharp$ such that if $v_\star>v_\sharp$, then there exists $T_0\in\R$ and $n_0\in\N$ such that for all $n\geq n_0$ the following property is satisfied for any $t_0\in[T_0,T^n]$. \\ If for all $t\in[t_0,T^n]$ we have
\begin{align*}
\norm*{\begin{pmatrix}u^n_1(t)\\u^n_2(t)\end{pmatrix}-\begin{pmatrix}R_1(t)\\R_2(t)\end{pmatrix}}_{H^1\times H^1}& \leq 
 e^{-\sqrt{\omega_\star}v_\star t},\\
\intertext{then for all $t\in[t_0,T^n]$  we have}
\norm*{\begin{pmatrix}u^n_1(t)\\u^n_2(t)\end{pmatrix}-\begin{pmatrix}R_1(t)\\R_2(t)\end{pmatrix}}_{H^1\times H^1} &\leq 
\frac12 e^{-\sqrt{\omega_\star}v_\star t}.
\end{align*}
\end{lem}

Before going further, we indicate how Lemma~\ref{lem:bootstrap} is used to prove Proposition \ref{prop:uniform}. 

\begin{proof}[Proof of Proposition \ref{prop:uniform}]
Let $T_0$, $n_0$, $v_\sharp$ be given by Lemma~\ref{lem:bootstrap}, fix $n>n_0$ and assume $v_\star>v_\sharp$. Define
\[
t_\sharp :=\inf\{t_\dag\text{ such that \eqref{eq:estimate} holds for all }t\in[t_\dag,T^n]\}.
\]
From the local well-posedness theory we know that $t_\sharp <T^n$. We prove by contradiction that $t_\sharp =T_0$. Assume that $t_\sharp >T_0$. By Lemma~\ref{lem:bootstrap}, for all $t\in[t_\sharp ,T^n]$  we have
\begin{equation*}
\norm*{\begin{pmatrix}u^n_1(t)\\u^n_2(t)\end{pmatrix}-\begin{pmatrix}R_1(t)\\R_2(t)\end{pmatrix}}_{H^1\times H^1} \leq 
\frac12 e^{-\sqrt{\omega_\star}v_\star t}.
\end{equation*}
Therefore, by continuity of $(u^n_1,u^n_2)^\intercal$, there exists $t_\ddagger<t_\sharp $ such that \eqref{eq:estimate} holds on $[t_\ddagger,T^n]$, hence contradicting the minimality of $t_\sharp $. As a consequence, $t_\sharp =T_0$ and the proposition is proved.
\end{proof}

Before proving Lemma~\ref{lem:bootstrap}, we need some preparation. We will work for fixed $n$, hence dependency in $n$ will only be understood, except for $T^n$. In particular, we shall denote $u^n_1$ by $u_1$, etc. Let $(\eps_1,\eps_2)^\intercal\in\hu\times\hu$ be such that
\begin{equation}\label{eq:def-eps}
\begin{pmatrix}
u_1\\u_2\end{pmatrix}=\begin{pmatrix}
R_1\\R_2\end{pmatrix}+\begin{pmatrix}
\eps_1\\\eps_2\end{pmatrix}.
\end{equation}
Take $t_0<T^n$ and assume the following \emph{bootstrap hypothesis}:
\begin{equation}\label{eq:bootstrap}
\norm*{\begin{pmatrix}
\eps_1(t)\\\eps_2(t)\end{pmatrix}}_{H^1\times H^1} \leq 
 e^{-\sqrt{\omega_\star}v_\star t}\quad\text{for all }t\in[t_0,T^n].
\end{equation}

For $j=1,2$, we denote by $S_j$ and $H_j$ the functionals defined for the solitary wave  $R_j$ in the same way as $S_0$ and $H_0$ were for $R_0$ in \eqref{eq:def-S} and \eqref{eq:def-H}. Note that, conversely to what was happening in the works \cite{CoLe11,CoMaMe11,MaMe06}, we do not need to localize the functionals around each solitary wave, since in our case the coupling will act as a localizing factor.
Let $\mathcal S$ be the functional defined for $(w_1,w_2)^\intercal\in\hu\times\hu$ by
\begin{equation}\label{eq:def-mathcal-S}
\mathcal S\begin{pmatrix}
w_1\\w_2\end{pmatrix}
:=
S_1(w_1)+S_2(w_2).
\end{equation}
and $\mathcal H$ be the functional defined for $(t,(\varpi_1,\varpi_2)^\intercal)\in\R\times\hu\times\hu$ by
\[
\mathcal H\left(t,\begin{pmatrix}
\varpi_1\\
\varpi_2\end{pmatrix}\right)
:=
H_1(t,\varpi_1)+H_2(t,\varpi_2).
\]
A direct consequence of Lemma~\ref{lem:coercivity} on $\mathcal H$ is the following result.

\begin{lem}[Vectorial Coercivity]\label{lem:coercivity-system}
There exists $c_\star>0$ such that 
for all $t\in\R$ and for all $(\varpi_1, \varpi_2)^\intercal\in\hu\times\hu$ we have:
\begin{equation}\label{eq:coercivity-property}
c_\star\norm*{\begin{pmatrix}
\varpi_1\\ \varpi_2\end{pmatrix}}^2_{H^1\times H^1} \leq\mathcal H\left(t,\begin{pmatrix}
\varpi_1\\ \varpi_2\end{pmatrix}\right)+\sum_{j=1,2}\sum_{k=1}^{\nu_j}\psld{\varpi_j}{\xi_j^k(t)}^2,
\end{equation}
where $(\xi_j^k)$ are given for for $j=1,2$ by Lemma \ref{lem:coercivity}.
\end{lem}

Note that the use of coercivity properties is reminiscent from the stability theory for standing waves of scalar nonlinear Schr\"odinger equation developed in \cite{CaLi82,GrShSt87,We85,We86}. However, in this theory, the functional equivalent to $\mathcal{S}$ is a conserved quantity, which is not the case for $\mathcal{S}$ (remark that $\mathcal S$ is build upon the conserved quantities of the scalar problem and not upon those of \eqref{eq:nls} given in \eqref{eq:conservE}-\eqref{eq:conservM}). However, we will still be able to estimate
the RHS of \eqref{eq:coercivity-property} thanks to an $L^2(\R^d)$-control (to deal with the scalar products) and thanks  to the fact that $\mathcal S$ is almost a conservation law (to deal with~$\mathcal H$).
\begin{lem}[$L^2(\R^d)$-control]\label{lem:L2-control}
Let $(\eps_1,\eps_2)^\intercal$ be given by \eqref{eq:def-eps} and assume \eqref{eq:bootstrap}. 
Then there exists $C>0$ independent of $v_\star$ such that for all $t\in [t_0,T^n]$ the following estimate holds:
\[
\norm*{\begin{pmatrix}
\eps_1(t)\\ \eps_2(t)\end{pmatrix}
}_{L^2\times L^2}
\leq
\frac{C}{\sqrt{\omega_\star}v_\star} e^{-\sqrt{\omega_\star}v_\star t}.
\]
\end{lem}
\begin{lem}[Almost Conservation Law]\label{lem:conservation}
	Assume \eqref{eq:bootstrap}. There exists $T_0>0$ depending only on $v_1,v_2$ such that 
	if $t_0>T_0$ then there exists $C>0$ independent of $n$ and of $v_\star$ such that for all $t\in [t_0,T^n]$ the following estimate holds:
	\begin{equation}\label{eq:conservation}
	\abs*{\mathcal S\begin{pmatrix}
	u_1(t)\\u_2(t)\end{pmatrix}-
	\mathcal S\begin{pmatrix}
	u_1(T^n)\\u_2(T^n)\end{pmatrix}}
	\leq \frac{C}{\sqrt{\omega_\star}v_\star} e^{-2\sqrt{\omega_\star}v_\star t}.
	\end{equation}
\end{lem}

Before showing Lemmas \ref{lem:L2-control} and \ref{lem:conservation}, we prove Lemma~\ref{lem:bootstrap}.

\begin{proof}[Proof of Lemma~\ref{lem:bootstrap}]
Let $(\eps_1,\eps_2)^\intercal$ be given by \eqref{eq:def-eps}, assume \eqref{eq:bootstrap} and assume also that $t_0>T_0$ where $T_0$ is given by Lemma \ref{lem:conservation}. 
Let $t\in[t_0,T^n]$.
By Lemma~\ref{lem:coercivity-system}, we have the following estimate
\begin{equation}\label{eq:coer}
c_\star\norm*{\begin{pmatrix}
\eps_1(t)\\ \eps_2(t)\end{pmatrix}}^2_{H^1\times H^1} \leq\mathcal H\left(t,\begin{pmatrix}
\eps_1(t)\\ \eps_2(t)\end{pmatrix}\right)+\sum_{j=1,2}\sum_{k=1}^{\nu_j}\psld{\eps_j(t)}{\xi_j^k(t)}^2.
\end{equation}
Using that $R_1$ and $R_2$ are critical points of $S_1$ and $S_2$, we have
\begin{multline}\label{eq:combine-1}
\mathcal S\begin{pmatrix}
u_1(t)\\u_2(t)\end{pmatrix}=\mathcal S\begin{pmatrix}
R_1(t)+\eps_1(t)\\R_2(t)+\eps_2(t)\end{pmatrix}=\\
\mathcal S\begin{pmatrix}
R_1(t)\\R_2(t)\end{pmatrix}+\mathcal H\left(t ,\begin{pmatrix}
\eps_1(t)\\ \eps_2(t)\end{pmatrix}\right)+O\left( \norm*{\begin{pmatrix}
\eps_1(t)\\ \eps_2(t)\end{pmatrix}}^3_{H^1\times H^1} \right).
\end{multline}
By Lemma~\ref{lem:conservation}, we have
	\begin{equation}\label{eq:combine-2}
	\abs*{\mathcal S\begin{pmatrix}
	u_1(t)\\u_2(t)\end{pmatrix}-
	\mathcal S\begin{pmatrix}
	u_1(T^n)\\u_2(T^n)\end{pmatrix}}
	\leq \frac{C}{\sqrt{\omega_\star}v_\star} e^{-2\sqrt{\omega_\star}v_\star t}.
	\end{equation}
By definition of $(u_1,u_2)^\intercal$ and since $\mathcal{S}$ is made of conserved quantities for $R_1$ and $R_2$, we have:
\begin{equation}\label{eq:combine-3}
\mathcal S\begin{pmatrix}
u_1(T^n)\\u_2(T^n)\end{pmatrix}=\mathcal S\begin{pmatrix}
R_1(T^n)\\R_2(T^n)\end{pmatrix}=\mathcal S\begin{pmatrix}
R_1(t)\\R_2(t)\end{pmatrix}.
\end{equation}
From the bootstrap assumption \eqref{eq:bootstrap} we have
\begin{equation}\label{eq:combine-4}
O\left( \norm*{\begin{pmatrix}
\eps_1(t)\\ \eps_2(t)\end{pmatrix}}^3_{H^1\times H^1} \right)=Ce^{-3\sqrt{\omega_\star}v_\star t}.
\end{equation}
Combining \eqref{eq:combine-1}-\eqref{eq:combine-4}, we infer that, possibly increasing $T_0$, we have:
\begin{equation}\label{eq:H1}
\abs*{\mathcal H\left(t ,\begin{pmatrix}
\eps_1(t)\\ \eps_2(t)\end{pmatrix}\right)}\leq \frac{C}{\sqrt{\omega_\star}v_\star} e^{-2\sqrt{\omega_\star}v_\star t}.
\end{equation}
Hence to control the $\hu\times\hu-$norm it remains to control the $L^2(\R^d)$ scalar products in the RHS of \eqref{eq:coer}. This is done using Lemma~\ref{lem:L2-control} and remembering that the $\xi_j^k$ are bounded in $L^2(\R^d)$:
\begin{equation}\label{eq:L2}
\sum_{j=1,2}\sum_{k=1}^{\nu_j}\psld{\eps_j(t)}{\xi_j^k(t)}^2
\leq C\norm*{\begin{pmatrix}
\eps_1(t)\\ \eps_2(t)\end{pmatrix}}_{L^2\times L^2}^2\leq \frac{C}{\sqrt{\omega_\star}v_\star} e^{-2\sqrt{\omega_\star}v_\star t}.
\end{equation}
Combining \eqref{eq:coer}, \eqref{eq:H1} and\eqref{eq:L2} we get
\[
\norm*{\begin{pmatrix}
\eps_1(t)\\ \eps_2(t)\end{pmatrix}}^2_{H^1\times H^1}\leq \frac{C}{\sqrt{\omega_\star}v_\star} e^{-2\sqrt{\omega_\star}v_\star t}.
\]
Therefore, there exists $v_\sharp$ such that if $v_\star>v_\sharp$ then we have
\[
\norm*{\begin{pmatrix}
\eps_1(t)\\ \eps_2(t)\end{pmatrix}}_{H^1\times H^1}\leq \frac{1}{2} e^{-\sqrt{\omega_\star}v_\star t},
\]
which is the desired conclusion.
\end{proof}

The following estimate on the interaction of the  two solitary waves will be central in the proofs of Lemmas \ref{lem:L2-control} and \ref{lem:conservation}.

\begin{lem}[Solitary Waves Interaction]\label{lem:exponential-decay}
There exists $C>0$ depending on $\Phi_1,\Phi_2$, $\omega_1,\omega_2$, $\mu_1,\mu_2$, \emph{but not on } $v_1,v_2$ such that for all $x\in \R^d$ 
we have 
\begin{align*}
\norm[\Big]{|R_1(t)||R_2(t)|}_{L^2}&\leq
 Ce^{-\frac32\sqrt{\omega_\star}v_\star t},\\
\norm[\Big]{\big(|R_1(t)|+|\nabla R_1(t)|\big)\big(|R_2(t)|+|\nabla R_2(t)|\big)}_{L^2}&\leq
 C(1+|v_1|+|v_2|)^2e^{-\frac32\sqrt{\omega_\star}v_\star t}.
\end{align*}
\end{lem}

\begin{proof}
Take $0<\eta<1$. 
Each $\Phi_j$ verifies 
\[
|\Phi_j(x)|+|\nabla\Phi_j(x)|\leq C e^{-\eta|x|}, 
\]
where $C=C(\Phi_j)$. Using the definition \eqref{eq:soliton-j} of a solitary wave, we have for each $R_j$ the estimate
\[
|R_j(t,x)|+|\nabla R_j(t,x)|\leq  C(1+|v_j|) e^{-\eta\sqrt{\omega_j}|x-v_jt-x_j|},
\]
where $C=C(\Phi_j,\omega_j,\mu_j)$.
Therefore, 
\begin{multline*}
(|R_1(t,x)|+|\nabla R_1(t,x)|)(|R_2(t,x)|+|\nabla R_2(t,x)|)\\
	\leq  C(1+|v_1|+|v_2|)^2e^{-\eta\sqrt{\min_{j=1,2}\{\omega_j\}}(|x-v_1t-x_1|+|x-v_2t-x_2|)},
\end{multline*}
where $C$ depends on $\Phi_1,\Phi_2,$  $\omega_1,\omega_2$ and $\mu_1,\mu_2$.
Let $0<\delta<\eta$. Since 
\[
|(v_1-v_2)t|\leq |x-v_1t|+|x-v_2t|,
\]
we infer that
\begin{multline*}
(|R_1(t,x)|+|\nabla R_1(t,x)|)(|R_2(t,x)|+|\nabla R_2(t,x)|)\\
	\leq  C(1+|v_1|+|v_2|)^2 e^{-\delta\sqrt{\min_{j=1,2}\{\omega_j\}}(|x-v_1t-x_1|+|x-v_2t-x_2|)}\\
	\cdot e^{-(\eta-\delta)\sqrt{\min_{j=1,2}\{\omega_j\}}|(v_1-v_2)t|},
\end{multline*}
where now $C$ depends also on $x_1$, $x_2$.
Choosing $\eta=\frac78$, $\delta=\frac18$ and remembering that $\omega_\star=\frac{1}{4}\min\{\omega_1,\omega_2\}$ and $v_\star=|v_1-v_2|$, we obtain
\begin{multline*}
(|R_1(t,x)|+|\nabla R_1(t,x)|)(|R_2(t,x)|+|\nabla R_2(t,x)|)\\
	\leq  C(1+|v_1|+|v_2|)^2e^{-\frac12\sqrt{\omega_\star}(|x-v_1t-x_1|+|x-v_2t-x_2|)}
e^{-\frac32\sqrt{\omega_\star}v_\star t}.
\end{multline*}
Taking the $L^2(\R^d)-$norm and using Cauchy-Schwartz inequality, we get
\begin{multline*}
\norm[\Big]{\big(|R_1(t)|+|\nabla R_1(t)|\big)\big(|R_2(t)|+|\nabla R_2(t)|\big)}_{L^2}\\
	\leq  C(1+|v_1|+|v_2|)^2e^{-\frac32\sqrt{\omega_\star}v_\star t}\norm{e^{-\frac12\sqrt{\omega_\star}|x|} }_{L^2}\\
	\leq   C(1+|v_1|+|v_2|)^2e^{-\frac32\sqrt{\omega_\star}v_\star t},
\end{multline*}
which is the desired conclusion.
\end{proof}

To prove the $L^2(\R^d)$-control Lemma \ref{lem:L2-control}, as in \cite{CoLe11} we adopt the following strategy. We first write the system satisfied by $(\eps_1,\eps_2)^\intercal$. Then, we differentiate in time the $L^2(\R^d)$-masses of $\eps_1$ and $\eps_2$, and estimate the result with  $e^{-2\sqrt{\omega_\star}v_\star t}$. Integrating in time finally allows us to gain the extra factor $\frac{1}{\sqrt{\omega_\star}v_\star}$.

\begin{proof}[Proof of Lemma~\ref{lem:L2-control}]
The couple
$(\varepsilon_1,\varepsilon_2)^\intercal$ satisfies the equation
\[
i\partial_t\begin{pmatrix}
\varepsilon_1\\ \varepsilon_2\end{pmatrix}+
\mathcal L \begin{pmatrix}
\varepsilon_1\\ \varepsilon_2\end{pmatrix}+
\mathcal N\begin{pmatrix}
\varepsilon_1\\\varepsilon_2\end{pmatrix}+
\mathcal F=0
\]
where $\mathcal L$ denote the linear part in $(\varepsilon_1, \varepsilon_2)^\intercal$, $\mathcal N$ the nonlinear part and $\mathcal F$ the source term. Precisely, we set
\[
\mathcal L \begin{pmatrix}
\varepsilon_1\\\varepsilon_2\end{pmatrix}:=
\begin{pmatrix}
L_1(\varepsilon_1,\varepsilon_2)\\L_2(\varepsilon_1,\varepsilon_2)\end{pmatrix},
\qquad
\mathcal N \begin{pmatrix}
\varepsilon_1\\\varepsilon_2\end{pmatrix}
:=\begin{pmatrix}
N_1(\varepsilon_1,\varepsilon_2)
\\
N_2(\varepsilon_1,\varepsilon_2)
\end{pmatrix},
\qquad
\mathcal F:=
\beta \begin{pmatrix}
|R_1|^2R_2\\|R_2|^2R_1\end{pmatrix},
\]
where 
\begin{multline*}
\begin{pmatrix}
L_1(\varepsilon_1,\varepsilon_2)\\L_2(\varepsilon_1,\varepsilon_2)
\end{pmatrix}=
\begin{pmatrix}\Delta \varepsilon_1+(2\mu_1|R_1|^2+\beta |R_2|^2)\varepsilon_1
+\mu_1R_1^2\bar \varepsilon_1+
\beta(R_1\bar R_2\eps_2+R_1R_2\bar\eps_2)\\
\Delta \varepsilon_2+(2\mu_2|R_2|^2+\beta |R_1|^2)\varepsilon_2
+\mu_2R_2^2\bar \varepsilon_2+
\beta(\bar R_1R_2\eps_1+R_1R_2\bar\eps_1)
\end{pmatrix},
\\
\shoveleft{
\begin{pmatrix}
N_1(\varepsilon_1,\varepsilon_2)
\\
N_2(\varepsilon_1,\varepsilon_2)
\end{pmatrix}
=
\begin{pmatrix}
\mu_1\left(\bar R_1\varepsilon_1^2+2R_1|\varepsilon_1|^2+|\varepsilon_1|^2\varepsilon_1\right)
\\
\mu_2\left(\bar R_2\varepsilon_2^2+2R_2|\varepsilon_2|^2+|\varepsilon_2|^2\varepsilon_2\right)
\end{pmatrix}}\\+\beta 
\begin{pmatrix}
R_2\bar\varepsilon_2\varepsilon_1+\bar R_2\varepsilon_2\varepsilon_1+R_1|\varepsilon_2|^2
+|\varepsilon_2|^2\varepsilon_1
\\
R_1\bar\varepsilon_1\varepsilon_2+\bar R_1\varepsilon_1\varepsilon_2+R_2|\varepsilon_1|^2
+|\varepsilon_1|^2\varepsilon_2
\end{pmatrix}.
\end{multline*}
We make the computations for $\eps_1$, the case of $\eps_2$ being exactly symmetric. 
\begin{multline}\label{eq:time}
\frac{\partial}{\partial t}M(\eps_1)=\frac{1}{2}\frac{\partial}{\partial t}\left(\norm*{\varepsilon_1(t)}_{L^2}^2\right)=
-\Im \int_{\R^d} (L_1(\varepsilon_1,\varepsilon_2)\bar{\varepsilon}_1+ N_1(\varepsilon_1,\varepsilon_2)\bar{\varepsilon}_1\\+\beta|R_1|^2R_2\bar{\varepsilon}_1)dx.
\end{multline}
Using the bootstrap assumption \eqref{eq:bootstrap}, we immediately obtain the following estimate:
\begin{align}
\left| \Im \int_{\R^d} L_1(\varepsilon_1,\varepsilon_2)\bar{\varepsilon}_1dx \right|
&=
\left|\Im\int_{\R^d} 
\mu_1R_1^2\bar \varepsilon_1^2+
\beta(R_1\bar R_2\bar\eps_1\eps_2+R_1R_2\bar\eps_1\bar\eps_2)
dx\right|,
\notag\\
&\leq C(\norm{R_1}^2_{L^\infty}+\norm{R_2}_{L^\infty}^2)(\norm{\varepsilon_1}^{2}_{H^1} +\norm{\varepsilon_2}^2_{H^1}),\notag\\
&\leq C e^{-2\sqrt{\omega_{\star}}v_{\star}t}.\label{eq:linear}
\end{align}
Here, and in the rest of the proof, the constant $C$ may depend on $\beta, \mu_1, \mu_2,$ $\Phi_1, \Phi_2,x_1,x_2$, but \emph{not} on $v_1,v_2$. This is due to the fact that $\norm{R_j}_{L^\infty}=\sqrt{\frac{\omega_j}{ \mu_j}}\norm{\Phi_j}_{L^\infty}$  for $j=1,2$. We consider now the nonlinear part. Since $d\leq 3$ we have the embedding  of $H^1(\mathbb R^d)$ into $L^3(\mathbb R^d)$ and $L^4(\mathbb R^d)$ and therefore we can prove that
\begin{align}\label{eq:nonlinear}
\abs*{ \Im \int_{\R^d} N_1(\varepsilon_1,\varepsilon_2)\bar{\varepsilon}_1dx} 
	&\leq\abs*{ \int_{\R^d} \mu_1\left(\bar R_1\varepsilon_1^2+2R_1|\varepsilon_1|^2+|\varepsilon_1|^2\varepsilon_1\right)\bar{\varepsilon}_1dx}\notag\\
	&\qquad+\abs*{ \int_{\R^d}\beta|(R_2\bar\varepsilon_2\varepsilon_1+\bar R_2\varepsilon_2\varepsilon_1+R_1|\varepsilon_2|^2
+|\varepsilon_2|^2\varepsilon_1)\bar{\varepsilon}_1dx}\notag\\
	&\leq C(\norm{R_1}_{L^\infty}+\norm{R_2}_{L^\infty})(\norm{\eps_1}^2_{H^1}\norm{\eps_2}^2_{H^1}+\norm{\eps_1}^3_{H^1}+\norm{\eps_1}^4_{H^1})  \notag\\
	&\leq C e^{-3\sqrt{\omega_{\star}}v_{\star}t}.
\end{align}
Last, in order to estimate the source term, we need to use also Lemma~\ref{lem:exponential-decay} in combinaison with the bootstrap assumption \eqref{eq:bootstrap}.
\begin{equation}\label{eq:source}
\left|\Im \int|R_2|^2R_1\bar{\varepsilon}_1\right|\leq 
\norm{R_2}_{L^\infty}\norm*{|R_1||R_2|}_{L^2}\norm{\varepsilon_1}_{H^1}\leq 
Ce^{-\frac{5}{2}\sqrt{\omega_{\star}}v_{\star}t},
\end{equation}
Combining \eqref{eq:time}-\eqref{eq:source} we get:
\[
\abs*{\frac{\partial}{\partial t}M(\eps_1)}\leq Ce^{-2\sqrt{\omega_{\star}}v_{\star}t}.
\]
Integrating in time and recalling that by definition we have $\eps_1(T^n)=0$, we obtain:
\[
M(\eps_1(t))
\leq \int^{T^n}_t \abs*{\frac{\partial}{\partial s}M(\eps_1(s))}ds
\leq \frac{C }{\sqrt{\omega_{\star}}v_{\star}}e^{-2\sqrt{\omega_{\star}}v_{\star}t},
\]
which is the desired conclusion for $\eps_1$. As already said, the calculations for $\eps_2$ are perfectly symmetric, hence the lemma is proved.
\end{proof}

Recall that $\mathcal S$ is build with scalar energies, masses and momentums. To prove Lemma \ref{lem:conservation}, the idea is, as for the proof of Lemma~\ref{lem:L2-control}, to differentiate in time the various quantities involved in $\mathcal S$ (see \eqref{eq:def-mathcal-S} and \eqref{eq:def-S}), control the result with $e^{-2\sqrt{\omega_\star}v_\star t}$  and then  integrate  to gain the extra factor $\frac{1}{\sqrt{\omega_\star}v_\star}$.

\begin{proof}[Proof of Lemma~\ref{lem:conservation}]
Since the scalar masses are conserved by the flow of \eqref{eq:nls} and $(u_1,u_2)^\intercal$ is a solution of \eqref{eq:nls}, it follows immediatly that 
\begin{equation}\label{eq:mass-control}
\abs*{M(u_1(t))-M(u_1(T^n))}
+\abs*{M(u_2(t))-M(u_2(T^n))}
=0.
\end{equation}

For the momentum part, we need to estimate
\[
\abs*{v_1\cdot \left(P(u_1(t))-P(u_1(T^n))\right)+v_2\cdot \left(P(u_2(t))-P(u_2(T^n))\right)}.
\]
In fact, since the total momentum \eqref{eq:conservP} is a conserved quantity, we have to estimate
\begin{equation}\label{eq:to-estimate}
\abs*{(v_1-v_2)\cdot \left(P(u_1(t))-P(u_1(T^n))\right)}=v_\star\abs*{P(u_1(t))-P(u_1(T^n))}.
\end{equation}
Hence we differentiating at time $t$ the scalar momentum $P_1$. Using the system \eqref{eq:nls} satisfied by  $(u_1,u_2)^\intercal$ and integrations by parts, we obtain
\[
\frac{\partial}{\partial t}P(u_1)=-\Im\int_{\R^d}\partial_t u_1\nabla \bar u_1dx=-\frac12\int_{\R^d}|u_2|^2\nabla|u_1|^2dx.
\]
We  recall that $u_1=R_1+\eps_1$ and $u_2=R_2+\eps_2$ and replace in the previous equation to get
\begin{multline}\label{eq:begin}
\frac{\partial}{\partial t}P(u_1)=-\frac12\int_{\R^d}
|R_2|^2\nabla|R_1|^2+
2|R_2|^2\nabla(\Re(\bar R_1\eps_1))+
|R_2|^2\nabla|\eps_1|^2\\+
2\Re(\bar R_2\eps_2)\nabla|R_1|^2+
4\Re(\bar R_2\eps_2)\nabla(\Re(\bar R_1\eps_1))+
2\Re(\bar R_2\eps_2)\nabla|\eps_1|^2\\+
|\eps_2|^2\nabla|R_1|^2+
2|\eps_2|^2\nabla(\Re(\bar R_1\eps_1)+
|\eps_2|^2\nabla|\eps_1|^2
dx.
\end{multline}
We treat the various products appearing differently depending on their order in $R_j$ and $\eps_j$. When there is a product of $R_1$ and $R_2$ or of their derivatives, we use Lemma~\ref{lem:exponential-decay}, as for the following term.
\begin{multline}
\abs*{\int_{\R^d}|R_2|^2\nabla|R_1|^2dx}
	\leq C\norm{(|R_1|+|\nabla R_1|)|R_2|}^2_{L^2}\\
	\leq C(1+|v_1|+|v_2|)^4e^{-3\sqrt{\omega_\star}v_\star t}.
\end{multline}
To deal with the $\eps_j$, we use the bootstrap assumption \eqref{eq:bootstrap}. With the help of Cauchy-Schwartz and H\"older inequalities and Sobolev embeddings, we get
\begin{equation}
\abs*{\int_{\R^d}|\eps_2|^2\nabla|\eps_1|^2dx}
	\leq C\norm{\nabla\eps_1}_{L^2}\norm{\eps_1}_{L^6}\norm{\eps_2}_{L^6}^2\leq Ce^{-4\sqrt{\omega_\star}v_\star t}.
\end{equation}
We possibly combine the two arguments as follows.
\begin{multline}
\abs*{\int_{\R^d}\Re(\bar R_2\eps_2)\nabla(\Re(\bar R_1\eps_1))dx}\\
	\leq \norm{|R_2|(|R_1|+|\nabla R_1|)}_{L^2}\norm{|\eps_2|(|\eps_1|+|\nabla\eps_1|)}_{L^2}\\
	\leq C(1+|v_1|+|v_2|)^2e^{-\frac72\sqrt{\omega_\star}v_\star t}.
\end{multline}
When there is an extra $R_j$ that we cannot use with Lemma~\ref{lem:exponential-decay}, we just take its $L^\infty(\R^d)$-norm:
\begin{multline}
\abs*{\int_{\R^d}|R_2|^2\nabla(\Re(\bar R_1\eps_1))dx}
+\abs*{\int_{\R^d}\Re(\bar R_2\eps_2)\nabla|R_1|^2dx}\\
	\leq C(\norm{R_1}_{L^\infty}+\norm{R_2}_{L^\infty})\norm{(|R_1|+|\nabla R_1|)|R_2|}_{L^2}(\norm{|\eps_1|+|\nabla\eps_1|}_{L^2}+\norm{\eps_2}_{L^2})\\
	\leq  C(1+|v_1|+|v_2|)^2e^{-\frac52\sqrt{\omega_\star}v_\star t}.
\end{multline}
The following estimate is obtained with similar arguments:
\begin{equation}\label{eq:beforeIBP}
\abs*{\int_{\R^d}\Re(\bar R_2\eps_2)\nabla|\eps_1|^2dx} 
	\leq \norm{R_2}_{L^\infty}\norm{\eps_1}_{L^4}\norm{\nabla\eps_1}_{L^2}\norm{\eps_2}_{L^4}\leq Ce^{-3\sqrt{\omega_\star}v_\star t}.
\end{equation}
After an integration by parts, the next product can be treated as in \eqref{eq:beforeIBP} 
\begin{equation}
\abs*{\int_{\R^d}|\eps_2|^2\nabla(\Re(\bar R_1\eps_1))dx}=
	\abs*{\int_{\R^d}\nabla|\eps_2|^2(\Re(\bar R_1\eps_1)dx}
	\leq Ce^{-3\sqrt{\omega_\star}v_\star t}.
\end{equation}
Before estimating the remaining two terms, we make a remark about $\norm{\nabla|R_j|^2}_{L^\infty}$. From the definition of a solitary wave \eqref{eq:soliton-j}, we have
\begin{multline*}
\nabla(|R_j|^2)=\frac{\omega_j}{\mu_j}\nabla{|\Phi_j\big(\sqrt{\omega_j} (x-v_jt-x_j)\big)|^2}=\\
\frac{2\omega_j^\frac32}{\mu_j} \Re\Big(\bar\Phi_j\big(\sqrt{\omega_j} (x-v_jt-x_j)\big)\nabla\Phi_j\big(\sqrt{\omega_j} (x-v_jt-x_j)\big)\Big).
\end{multline*}
This implies that 
\[
\norm{\nabla|R_j|^2}_{L^\infty}\leq \frac{2\omega_j^{\frac32}}{\mu_j}\norm{\Phi_j}_{L^\infty}\norm{\nabla \Phi_j}_{L^\infty},
\]
and in particular $\norm{\nabla|R_j|^2}_{L^\infty}$ \emph{does not depend on $v_j$}.
 We can now write
\[
\abs*{\int_{\R^d}|\eps_2|^2\nabla|R_1|^2dx}
	\leq \norm{\nabla|R_1|^2}_{L^\infty}\norm{\eps_2}_{L^2}^2\leq C\norm{\eps_2}_{L^2}^2,\\
\]
Here, if we use directly the bootstrap assumption \eqref{eq:bootstrap}, we will miss the correct estimate by a factor $\frac{1}{v_\star}$ because of the $v_\star$ appearing in \eqref{eq:to-estimate}. However, remembering that we already improved \eqref{eq:bootstrap} at the $L^2(\R^d)$-level in Lemma~\ref{lem:L2-control}, we can conclude that:
\begin{equation}
\abs*{\int_{\R^d}|\eps_2|^2\nabla|R_1|^2dx}
	\leq \frac{C}{\sqrt{\omega_\star}v_\star}e^{-2\sqrt{\omega_\star}v_\star}.
\end{equation}
The last term is treated in a similar fashion after an integration by parts.
\begin{equation}\label{eq:last}
\abs*{\int_{\R^d}|R_2|^2\nabla|\eps_1|^2dx}
	=\abs*{\int_{\R^d}\nabla|R_2|^2|\eps_1|^2dx}
	\leq \frac{C}{\sqrt{\omega_\star}v_\star}e^{-2\sqrt{\omega_\star}v_\star}.
\end{equation}
Take now $T_0$ large enough so that $v_\star(1+|v_1|+|v_2|)^4e^{-\frac{1}{2}\sqrt{\omega_\star}v_\star T_0}<1$. With this assumption and the fact that $t_0>T_0$, we can now combine \eqref{eq:begin}-\eqref{eq:last}, and argue in the same fashion for the scalar momentum of $u_2$, to finally find:
\begin{equation}\label{eq:momentum}
\abs*{v_\star\frac{\partial}{\partial t}P(u_1)}\leq C e^{-2\sqrt{\omega_\star}v_\star t}
\end{equation}
for $C$ depending on $\Phi_1,\Phi_2$, $\omega_1,\omega_2$, $\mu_1,\mu_2$ but not on $v_1,v_2$.
Therefore, we obtain the following control on scalar momentums
\begin{multline}\label{eq:momentum-control}
\abs*{v_1\cdot \left(P(u_1(t))-P(u_1(T^n))\right)+v_2\cdot \left(P(u_2(t))-P(u_2(T^n))\right)}\\
	=v_\star\abs*{P(u_1(t))-P(u_1(T^n))}
	\leq  \int_t^{T^n}\abs*{v_\star\frac{\partial}{\partial s}P(u_1(s))}ds\\
	\leq \frac{C}{\sqrt{\omega_\star}v_\star} e^{-2\sqrt{\omega_\star}v_\star t}.
\end{multline}

Now, we treat the energy part. The direct approach consisting in trying to differentiate in time the energies $E(u_j,\mu_j)$ and then argue as for the momentums is bound to fail because of the appearance of terms like
\[
\int_{\R^d}\Im(\eps_1\nabla\bar\eps_1)\Re(\eps_2\nabla\bar \eps_2)dx,
\]
which, unless $d=1$, we cannot treat with an $H^1(\R^d)$-information like \eqref{eq:bootstrap}.  However, if we use the conservation of the total energy $\mathcal E$ we remark that:
\begin{multline}\label{eq:energy-trick}
E(u_1(t),\mu_1)-E(u_1(T^n),\mu_1)+E(u_2(t),\mu_2)-E(u_2(T^n),\mu_2)\\=\mathcal E \begin{pmatrix}u_1(t)\\u_2(t)\end{pmatrix}-\mathcal E \begin{pmatrix}u_1(T^n)\\u_2(T^n)\end{pmatrix}-\beta\int_{\R^d}(|u_1(t)|^2|u_2(t)|^2-|u_1(T^n)|^2|u_2(T^n)|^2)dx\\
=-\beta\int_{\R^d}(|u_1(t)|^2|u_2(t)|^2-|u_1(T^n)|^2|u_2(T^n)|^2)dx.
\end{multline}
Therefore, it is enough to prove that 
\begin{equation}\label{eq:coupling}
\int_{\R^d}\big(|u_1(t)|^2|u_2(t)|^2+|u_1(T^n)|^2|u_2(T^n)|^2\big)dx\leq \frac{C}{\sqrt{\omega_\star}v_\star}e^{-2\sqrt{\omega_\star}v_\star t}.
\end{equation}
To obtain \eqref{eq:coupling}, we do not differentiate  in time the LHS but instead we try to obtain the estimate directly. 
First note that by definition of $(u_1,u_2)^\intercal$ and Lemma \ref{lem:exponential-decay} we have 
\begin{multline*}
\int_{\R^d}|u_1(T^n)|^2|u_2(T^n)|^2dx=\int_{\R^d}|R_1(T^n)|^2|R_2(T^n)|^2dx\\
	\leq Ce^{-3\sqrt{\omega_\star}v_\star T^n}\leq  Ce^{-3\sqrt{\omega_\star}v_\star t}.
\end{multline*}
As before, for the other part, we replace $u_j$ by $R_j+\eps_j$ and develop. 
\begin{multline}\label{eq:begin-energy}
\int_{\R^d}|u_1|^2|u_2|^2dx=\int_{\R^d}(
|R_1|^2|R_2|^2+
2|R_1|^2\Re(\bar R_2\eps_2)+
|R_1|^2|\eps_2|^2\\+
2\Re(\bar R_1\eps_1)|R_2|^2+
4\Re(\bar R_1\eps_1)\Re(\bar R_2\eps_2)+
2\Re(\bar R_1\eps_1)|\eps_2|^2\\+
|\eps_1|^2|R_2|^2+
|\eps_1|^2\Re(\bar R_2\eps_2)+
|\eps_1|^2|\eps_2|^2
)dx.
\end{multline}
The following estimates are obtained using the same arguments as in the momentum case, in particular  Lemma~\ref{lem:exponential-decay} and the bootstrap assumption \eqref{eq:bootstrap}.
\begin{align}
\abs*{\int_{\R^d}|R_1|^2|R_2|^2dx}
	&\leq Ce^{-3\sqrt{\omega_\star}v_\star t},\\
\abs*{\int_{\R^d}2|R_1|^2\Re(\bar R_2\eps_2)dx}
	&\leq C\norm{R_1}_{L^\infty}\norm{|R_1||R_2|}_{L^2}\norm{\eps_2}_{L^2}\leq Ce^{-\frac52\sqrt{\omega_\star}v_\star t},\\
\abs*{\int_{\R^d}\Re(\bar R_1\eps_1)|R_2|^2dx}
	&\leq C\norm{R_2}_{L^\infty}\norm{|R_1||R_2|}_{L^2}\norm{\eps_1}_{L^2}\leq Ce^{-\frac52\sqrt{\omega_\star}v_\star t},\\
\abs*{\int_{\R^d}\Re(\bar R_1\eps_1)\Re(\bar R_2\eps_2)dx}
	&\leq C\norm{|R_1||R_2|}_{L^2}\norm{|\eps_1||\eps_2|}_{L^2}\leq Ce^{-\frac72\sqrt{\omega_\star}v_\star t},\\
\abs*{\int_{\R^d}\Re(\bar R_1\eps_1)|\eps_2|^2dx}
	&\leq C\norm{R_1}_{L^\infty}\norm{\eps_1}_{L^2}\norm{\eps_2}^2_{L^4}\leq Ce^{-3\sqrt{\omega_\star}v_\star t},\\
\abs*{\int_{\R^d}|\eps_1|^2\Re(\bar R_2\eps_2)dx}
	&\leq  \norm{R_2}_{L^\infty}\norm{\eps_1}_{L^4}^2\norm{\eps_2}_{L^2}\leq  Ce^{-3\sqrt{\omega_\star}v_\star t},\\
\abs*{\int_{\R^d}|\eps_1|^2|\eps_2|^2dx}
	&\leq  \norm{\eps_1}_{L^4}^2\norm{\eps_2}^2_{L^4}\leq  Ce^{-4\sqrt{\omega_\star}v_\star t}.
\end{align}
We need an extra argument for the two remaining terms. Indeed, we have
\begin{align*}
\abs*{\int_{\R^d}|\eps_1|^2|R_2|^2dx}
	&\leq  \norm{R_2}_{L^\infty}^2\norm{\eps_1}_{L^2}^2\leq C \norm{\eps_1}_{L^2}^2,\\
\abs*{\int_{\R^d}|R_1|^2|\eps_2|^2dx}
	&\leq \norm{R_1}_{L^\infty}^2\norm{\eps_2}_{L^2}^2\leq C \norm{\eps_2}_{L^2}^2.
\end{align*}
As for the momemtum part, if we use  \eqref{eq:bootstrap} here, we miss the correct estimate by a factor $\frac{1}{\sqrt{\omega_\star}v_\star}$. However, using Lemma~\ref{lem:L2-control}, we can conclude that:
\begin{align}
\abs*{\int_{\R^d}|\eps_1|^2|R_2|^2dx}
	&\leq\frac{C}{\sqrt{\omega_\star}v_\star}e^{-2\sqrt{\omega_\star}v_\star t},\\
\abs*{\int_{\R^d}|R_1|^2|\eps_2|^2dx}
	&\leq \frac{C}{\sqrt{\omega_\star}v_\star}e^{-2\sqrt{\omega_\star}v_\star t}.\label{eq:last-energy}
\end{align}
Putting together \eqref{eq:begin-energy}-\eqref{eq:last-energy} and assuming $T_0$ large enough implies the desired estimate \eqref{eq:coupling}.

To conclude the proof, we combine \eqref{eq:mass-control}, \eqref{eq:momentum-control}, \eqref{eq:energy-trick}, \eqref{eq:coupling} to obtain \eqref{eq:conservation}.
\end{proof}

\section{Compactness of the sequence of initial data}\label{sec:compactness}

In this section, we prove Proposition \ref{prop:compactness}. The proof is similar to the one given in \cite{CoLe11,MaMe06} and we repeat it here for the sake of completness. We again use the superscript $n$ to indicate the dependency in $n$.

From Proposition \ref{prop:uniform}, we know that $(u^n_1(T_0),u^n_2(T_0))^\intercal$ is bounded  in $\hu\times\hu$. Hence there exist $(u^0_1,u^0_2)^\intercal\in \hu\times\hu$ such that 
\begin{equation}\label{eq:weakconv}
\begin{pmatrix}u^n_1(T_0)\\u^n_2(T_0)\end{pmatrix}\xrightharpoonup{H^1}\begin{pmatrix}u^0_1\\u^0_2\end{pmatrix}.
\end{equation}
We now prove that convergence in \eqref{eq:weakconv} holds also strongly in $L^2(\R^d)\times L^2(\R^d)$, the result of 
Proposition \ref{prop:compactness} then readily following by interpolation. Take $\delta>0$, let $n$ be large enough and let $T_\delta\in[T_0,T^n]$ be such that $e^{-\sqrt{\omega_\star}v_\star t}<\sqrt{\frac{\delta}{4}}$. Then, by Proposition \ref{prop:uniform},
\begin{equation}\label{eq:delta-est}
\norm*{\begin{pmatrix}u^n_1(T_\delta)\\u^n_2(T_\delta)\end{pmatrix}-\begin{pmatrix}R_1(T_\delta)\\R_2(T_\delta)\end{pmatrix}}_{H^1\times H^1}\leq \sqrt{\frac{\delta}{4}}.
\end{equation}
Take $\rho_\delta>0$ such that 
\begin{equation}\label{eq:rho}
\int_{|x|>\rho_\delta}|R_1(T_\delta)|^2+|R_2(T_\delta)|^2dx\leq \frac{\delta}{4}.
\end{equation}
Then we infer from \eqref{eq:delta-est} that 
\[
\int_{|x|>\rho_\delta}|u^n_1(T_\delta)|^2+|u^n_2(T_\delta)|^2dx\leq \frac{\delta}{2}.
\]
Our goal is transfer this smallness up to $T_0$. Let $\tau:\R\to\R$ be a $\mathcal C^1$ cut-off function  such that 
\[
\tau(s)=0\text{ for }s<0,\quad\tau(s)=1\text{ for }s>1,\quad \tau(s)\in[0,1]\text{ for }s\in\R, \quad\norm{\tau'}_{L^\infty}\leq 2.
\]
Take $\kappa_\delta>0$ and define
\[
V(t):=\frac{1}{2}\int_{\R^d}\left(|u^n_1|^2+|u^n_2|^2\right)\tau\left(\frac{|x|-\rho_\delta}{\kappa_\delta}\right)dx.
\]
Then we have
\[
V'(t)=\Re\int_{\R^d}\left(\bar u^n_1\partial_tu^n_1+\bar u^n_2\partial_t u^n_2\right)\tau\left(\frac{|x|-\rho_\delta}{\kappa_\delta}\right)dx.
\]
Using the equation satisfied by $u_1$ and after an integration by part, we obtain:
\begin{align}
\Re\int_{\R^d}\bar u^n_1\partial_tu^n_1\tau\left(\frac{|x|-\rho_\delta}{\kappa_\delta}\right)dx&=
\Im\int_{\R^d}\bar u^n_1\Delta u^n_1\tau\left(\frac{|x|-\rho_\delta}{\kappa_\delta}\right)dx,\nonumber\\
&=\frac{1}{\kappa_\delta}\Im\int_{\R^d}\bar u^n_1\frac{x}{|x|}\cdot\nabla u^n_1\tau'\left(\frac{|x|-\rho_\delta}{\kappa_\delta}\right)dx.\label{eq:kappa}
\end{align}
From Proposition \ref{prop:uniform} we know that there exists $n_0$ such that
\[
\sup_{n>n_0}\sup_{t\in[T_0,T^n]}\norm*{\begin{pmatrix}u^n_1(t)\\u^n_2(t)\end{pmatrix}}_{H^1\times H^1}\leq 1.
\]
Therefore, we infer from \eqref{eq:kappa} and similar computations for $u_2$ that 
\[
|V'(t)|\leq \frac{1}{\kappa_\delta}.
\]
Choose now $\kappa_\delta$ such that $\frac{T_\delta-T_0}{\kappa_\delta}<\frac{\delta}{2}$. Then 
\begin{equation}\label{eq:V}
V(T_0)-V(T_\delta)=\int_{T_\delta}^{T_0}V'(t)dt\leq \frac{T_\delta-T_0}{\kappa_\delta}\leq \frac{\delta}{2}.
\end{equation}
Set $r_\delta:=\kappa_\delta+\rho_\delta$ (note that $r_\delta$ is independant of $n$). Since from \eqref{eq:rho} and the definition of $\tau$ we have $V(T_\delta)< \frac{\delta}{2}$, we deduce from \eqref{eq:V} that
\[
\int_{|x|>r_\delta}|u^n_1(T_0)|^2+|u^n_2(T_0)|^2dx\leq V(T_0)\leq \delta.
\]
Therefore the sequence $(u^n_1(T_0),u^n_2(T_0))^\intercal$ is $L^2(\R^d)\times L^2(\R^d)$ compact, which concludes the proof.

 \bibliographystyle{abbrv}
 \bibliography{biblio}

\def\cprime{$'$} \def\cprime{$'$}
\begin{thebibliography}{10}

\bibitem{AbPrTr04}
M.~J. Ablowitz, B.~Prinari, and A.~D. Trubatch.
\newblock {\em Discrete and continuous nonlinear {S}chr\"odinger systems},
  volume 302 of {\em London Mathematical Society Lecture Note Series}.
\newblock Cambridge University Press, Cambridge, 2004.

\bibitem{Ab07}
G.~Agrawal.
\newblock {\em Nonlinear fiber optics}.
\newblock Optics and Photonics. Academic Press, 2007.

\bibitem{AmCo07}
A.~Ambrosetti and E.~Colorado.
\newblock Standing waves of some coupled nonlinear {S}chr\"odinger equations.
\newblock {\em J. Lond. Math. Soc. (2)}, 75(1):67--82, 2007.

\bibitem{BeLi83-1}
H.~Berestycki and P.-L. Lions.
\newblock Nonlinear \ scalar \ field\ equations {I}.
\newblock {\em Arch. Ration. Mech. Anal.}, 82:313--346, 1983.

\bibitem{BeLi83-2}
H.~Berestycki and P.-L. Lions.
\newblock Nonlinear scalar field equations {II}.
\newblock {\em Arch. Ration. Mech. Anal.}, 82(4):347--375, 1983.

\bibitem{Ca03}
T.~Cazenave.
\newblock {\em Semilinear {S}chr\"odinger equations}.
\newblock New York University -- Courant Institute, New York, 2003.

\bibitem{CaLi82}
T.~Cazenave and P.-L. Lions.
\newblock Orbital stability of standing waves for some nonlinear
  {S}chr\"odinger equations.
\newblock {\em Comm. Math. Phys.}, 85(4):549--561, 1982.

\bibitem{CaWe90}
T.~Cazenave and F.~B. Weissler.
\newblock The {C}auchy problem for the critical nonlinear {S}chr\"odinger
  equation in {$H^s$}.
\newblock {\em Nonlinear Anal.}, 14(10):807--836, 1990.

\bibitem{CoCoOh09}
M.~Colin, T.~Colin, and M.~Ohta.
\newblock Stability of solitary waves for a system of nonlinear {S}chr\"odinger
  equations with three wave interaction.
\newblock {\em Ann. Inst. H. Poincar\'e Anal. Non Lin\'eaire},
  26(6):2211--2226, 2009.

\bibitem{CoGaYa09}
C.~Cort{\'a}zar, M.~Garc{\'{\i}}a-Huidobro, and C.~S. Yarur.
\newblock On the uniqueness of the second bound state solution of a semilinear
  equation.
\newblock {\em Ann. Inst. H. Poincar\'e Anal. Non Lin\'eaire},
  26(6):2091--2110, 2009.

\bibitem{CoGaYa11}
C.~Cort{\'a}zar, M.~Garc{\'{\i}}a-Huidobro, and C.~S. Yarur.
\newblock On the uniqueness of sign changing bound state solutions of a
  semilinear equation.
\newblock {\em Ann. Inst. H. Poincar\'e Anal. Non Lin\'eaire}, 28(4):599--621,
  2011.

\bibitem{CoLe11}
R.~C{\^o}te and S.~Le~Coz.
\newblock High-speed excited multi-solitons in nonlinear {S}chr\"odinger
  equations.
\newblock {\em J. Math. Pures Appl. (9)}, 96(2):135--166, 2011.

\bibitem{CoMaMe11}
R.~C{\^o}te, Y.~Martel, and F.~Merle.
\newblock Construction of multi-soliton solutions for the {$L^2$}-supercritical
  g{K}d{V} and {NLS} equations.
\newblock {\em Rev. Mat. Iberoam.}, 27(1):273--302, 2011.

\bibitem{Es97}
B.~Esry, C.~H. Greene, J.~P. Burke~Jr, and J.~L. Bohn.
\newblock Hartree-fock theory for double condensates.
\newblock {\em Physical Review Letters}, 78:3594--3597, 1997.

\bibitem{GrShSt87}
M.~Grillakis, J.~Shatah, and W.~A. Strauss.
\newblock Stability theory of solitary waves in the presence of symmetry {I}.
\newblock {\em J. Funct. Anal.}, 74(1):160--197, 1987.

\bibitem{La80}
G.~L. Lamb, Jr.
\newblock {\em Elements of soliton theory}.
\newblock John Wiley \& Sons Inc., New York, 1980.

\bibitem{LiWe05}
T.-C. Lin and J.~Wei.
\newblock Ground state of {$N$} coupled nonlinear {S}chr\"odinger equations in
  {$\mathbf R^n$}, {$n\leq 3$}.
\newblock {\em Comm. Math. Phys.}, 255(3):629--653, 2005.

\bibitem{Li86}
P.-L. Lions.
\newblock Solutions complexes d'\'equations elliptiques semilin\'eaires dans
  {${\bf R}\sp N$}.
\newblock {\em C. R. Acad. Sci. Paris S\'er. I Math.}, 302(19):673--676, 1986.

\bibitem{MaMoPe06}
L.~A. Maia, E.~Montefusco, and B.~Pellacci.
\newblock Positive solutions for a weakly coupled nonlinear {S}chr\"odinger
  system.
\newblock {\em J. Differential Equations}, 229(2):743--767, 2006.

\bibitem{MaMoPe10}
L.~A. Maia, E.~Montefusco, and B.~Pellacci.
\newblock Orbital stability property for coupled nonlinear {S}chr\"odinger
  equations.
\newblock {\em Adv. Nonlinear Stud.}, 10(3):681--705, 2010.

\bibitem{Ma74}
S.~V. Manakov.
\newblock {On the theory of two-dimensional stationary self-focusing of
  electromagnetic waves}.
\newblock {\em Journal of Experimental and Theoretical Physics}, 38, 1974.

\bibitem{MaMe06}
Y.~Martel and F.~Merle.
\newblock {Multi solitary waves for nonlinear {S}chr{\"o}dinger equations}.
\newblock {\em Ann. Inst. H. Poincar{\'e} Anal. Non Lin{\'e}aire},
  23(6):849--864, 2006.

\bibitem{Me90}
F.~Merle.
\newblock Construction of solutions with exactly {$k$} blow-up points for the
  {S}chr\"odinger equation with critical nonlinearity.
\newblock {\em Comm. Math. Phys.}, 129(2):223--240, 1990.

\bibitem{Oh96}
M.~Ohta.
\newblock Stability of solitary waves for coupled nonlinear {S}chr\"odinger
  equations.
\newblock {\em Nonlinear Anal.}, 26(5):933--939, 1996.

\bibitem{Sc86}
P.~C. Schuur.
\newblock {\em Asymptotic analysis of soliton problems}, volume 1232 of {\em
  Lecture Notes in Mathematics}.
\newblock Springer-Verlag, Berlin, 1986.

\bibitem{Si07}
B.~Sirakov.
\newblock Least energy solitary waves for a system of nonlinear {S}chr\"odinger
  equations in {$\mathbb R^n$}.
\newblock {\em Comm. Math. Phys.}, 271(1):199--221, 2007.

\bibitem{SuSu99}
C.~Sulem and P.-L. Sulem.
\newblock {\em The nonlinear {S}chr\"odinger equation}, volume 139 of {\em
  Applied Mathematical Sciences}.
\newblock Springer-Verlag, New York, 1999.
\newblock Self-focusing and wave collapse.

\bibitem{Ta06}
T.~Tao.
\newblock {\em Nonlinear dispersive equations}, volume 106 of {\em CBMS
  Regional Conference Series in Mathematics}.
\newblock Published for the Conference Board of the Mathematical Sciences,
  Washington, DC, 2006.
\newblock Local and global analysis.

\bibitem{TeVe09}
S.~Terracini and G.~Verzini.
\newblock Multipulse phases in {$k$}-mixtures of {B}ose-{E}instein condensates.
\newblock {\em Arch. Ration. Mech. Anal.}, 194(3):717--741, 2009.

\bibitem{WeWe08}
J.~Wei and T.~Weth.
\newblock Radial solutions and phase separation in a system of two coupled
  {S}chr\"odinger equations.
\newblock {\em Arch. Ration. Mech. Anal.}, 190(1):83--106, 2008.

\bibitem{We85}
M.~I. Weinstein.
\newblock Modulational stability of ground states of nonlinear {S}chr\"odinger
  equations.
\newblock {\em SIAM J. Math. Anal.}, 16:472--491, 1985.

\bibitem{We86}
M.~I. Weinstein.
\newblock Lyapunov stability of ground states of nonlinear dispersive evolution
  equations.
\newblock {\em Comm. Pure Appl. Math.}, 39(1):51--67, 1986.

\end{thebibliography}


\end{document}